\documentclass[a4paper,12pt]{article}
\usepackage{amssymb}
\usepackage{amsfonts}
\usepackage{latexsym}
\usepackage{amsmath}
\usepackage{amsthm}

\theoremstyle{plain}
\newtheorem{theorem}{Theorem}[section]
\newtheorem{lemma}[theorem]{Lemma}
\newtheorem{corollary}[theorem]{Corollary}
\newtheorem{proposition}[theorem]{Proposition}

\theoremstyle{definition}
\newtheorem*{definition}{Definition}

\theoremstyle{remark}
\newtheorem*{remark}{Remark}
\newtheorem*{notation}{Notation}

\numberwithin{equation}{section}

\def\C{\mathbb C}
\def\F{\mathbb F}
\def\K{\mathbb K}
\def\Z{\mathbb Z}
\def\T{\mathcal T}
\def\E{\mathcal E}

\def\R{\widehat R}
\def\L{\widehat L}
\def\u{\mathfrak {u}}

\def\e{\varepsilon}

\def\rank{\text{rank}}
\def\supp{\text{supp}}

\begin{document}

\title{Representations of Finite Unipotent Linear Groups \\ by the Method of Clusters}

\date{May 17, 2006}
\author{Ning Yan \\ \\ 
       Stanford University \\ \\
        e-mail: ning.now@gmail.com }

\maketitle

\section*{Abstract}

The general linear group $GL(n, \K)$ over a field $\K$
contains a particularly prominent subgroup $U(n, \K)$, consisting of all the
upper triangular unipotent elements. In this paper we
are interested in the case when $\K$ is the finite field $\F_q$,
and our goal is to better understand the representation
theory of $U(n,\F_q)$. The complete classification of
the complex irreducible representations of this group has long
been known to be a difficult task. The orbit method of Kirillov,
famous for its success when $\K$ has characteristic $0$, is
a natural source of intuition and conjectures, but in our case
the relation between coadjoint orbits and
complex representations is still a mystery. 
Here we introduce a natural variant of the orbit method, 
in which the central role is played by certain clusters of coadjoint orbits.
This ``method of clusters'' leads to the construction of a subring
in the representation ring of $U(n,\F_q)$ that is rich in structure but
pleasantly comprehensible. The cluster method also has
many of the major features one would expect from the 
philosophy of orbit method.

\section*{Introduction}

Over 40 years ago A. A. Kirillov \cite{Kirillov-1962} created the method of orbits, which not only
gives beautifully compact and complete answers, all in terms of coadjoint orbits,
to the main questions of representation theory for nilpotent Lie groups, but also has become a powerful tool and a 
source of intuition for the study of other groups of Lie type.

The upper triangular unipotent group $U(n)$ has been the ideal example for which the orbit method
works perfectly. However, there remains a tantalizing open question: The classification of the coadjoint orbits.

The present work starts with two simple observations. First, the coadjoint action of  $U(n, \F_q)$ on the
dual space of its Lie algebra can be extended to a double action.
Second, the group algebra of $U(n, \F_q)$ has a natural basis, the {\it Fourier basis}. The double 
action and the Fourier basis are connected through the regular representation of $U(n, \F_q)$ on its 
group algebra, and this connection leads to a rich theory of {\it cluster characters} [explain more]
with many of the major features one would expect from the philosophy of orbit method. These features
include an one-to-one correspondence between the cluster characters and the {\it coadjoint clusters},
a character formula expressing the character in terms of the cluster, and equivalence between character
operations (tensor product, induction and restriction) and cluster operations. A remarkable new feature is the
canonical decomposition of the regular representation into a direct sum of {\it cluster modules}.

The collection of all cluster characters form an orthogonal basis for the $\C$-linear
space of {\it cluster functions} on $U(n, \F_q)$. The collection of all $\Z$-linear combinations
of cluster characters form a subring in the representation ring of $U(n, \F_q)$. 
This subring is generated by the {\it primary} cluster characters, and each cluster
character has a unique {\it primary decomposition} into a tensor product of primary cluster characters.

Section \ref{discrete} is inspired by a paper of G.I.~Lehrer \cite{Lehrer74}. Lehrer was interested in the
{\it discrete series character} of $U(n, \F_q)$, which is by definition the common restriction of the 
discrete series characters of $GL(n, \F_q)$. He derived
a tensor product decomposition of this character, and then further decomposed each factor into a
multiplicity-free sum of irreducible components.
We noticed that these irreducible components are precisely the primary cluster characters of our theory.
This led us to the discovery of an explicit decomposition of the discrete series character into a
sum of cluster characters. Moreover, we were able to
construct a realization of the discrete series representation as a canonical submodule in the regular
representation of $U(n, \F_q)$.

To close this introduction we draw the reader's attention to the series of papers on the
representation theory of $U(n, \F_q)$, by C.A.M.~Andr\'e 
\cite{Carlos95a, Carlos95b, Carlos96, Carlos98, Carlos01, Carlos02, Carlos04}.
Throughout his papers Andr\'e makes the assumption $p\ge n$ in order to have
the benefit of a well-behaved exponential map. In this case an one-to-one correspondence between the
irreducible characters and the coadjoint orbits was established by Kazhdan. Andr\'e defines certain
``basic sums'' of  coadjoint orbits, and studies the characters associated to the orbit sums via the
above correspondence. Note that here the popular mode of thinking is to go from $\F_q$ to its algebraic
closure, so that the coadjoint orbits can be treated as geometric objects. It transpired that the orbit
sums of Andr\'e are precisely the coadjoint clusters, and the characters in his work are the cluster
characters. The main results of Andr\'e agree well with our own theory, while his approach is an impressive
tour de force in algebraic geometry.

\section{Basic notations}\label{notation}

Let $U(n,\K)$ be the group of upper triangular unipotent matrices over a field $\K$,
$\u(n,\K)$ be the $\K$-algebra of upper
triangular nilpotent matrices over $\K$, and 
$\u^*(n,\K)$ be the dual $\K$-linear space of
$\u(n,\K)$. An element in $\u^*(n,\K)$ is a 
$\K$-linear map from $\u(n,\K)$ to $\K$.

There are three natural ways for $U(n,\K)$ to act on $\u(n,\K)$.
These are defined by:
\begin{eqnarray*}
\textit{left action}  \qquad X & \mapsto & g\cdot X  \\
\textit{right action} \qquad X & \mapsto & X\cdot g \\
\textit{adjoint action}       \qquad X & \mapsto & g\cdot X\cdot g^{-1}
\end{eqnarray*}
where $g\in U(n,\K)$ and $X\in\u(n,\K)$.
The dot $(\,\cdot\,)$ means multiplication of matrices.
The left and right actions commute, and together they give a double 
action of $U(n,\K)$ on $\u(n,\K)$.
The {\it adjoint orbit} $C(X)$ and the {\it adjoint cluster} $K(X)$ of $X$ are
\begin{eqnarray*}
C(X) & = & \left\{\, g\cdot X\cdot g^{-1}:\quad g\in U(n,\K)\,\right\}, \\
K(X) & = & \left\{\, g\cdot X\cdot g': \quad g,\,g'\in U(n,\K)\, \right\},
\end{eqnarray*}
with $C(X)\subset K(X)$. Each adjoint cluster is a union of adjoint orbits.

Correspondingly we have the following actions of $U(n,\K)$ on
$\u^*(n,\K)$:
\begin{eqnarray*}
\textit{left action} \qquad (g*\lambda)(X) & = & \lambda(X\cdot g) \\
\textit{right action} \qquad (\lambda *g)(X) & = & \lambda(g\cdot X) \\
\textit{coadjoint action} \qquad\qquad \lambda^g(X) & = & \lambda(g^{-1}\cdot X\cdot g)
\end{eqnarray*}
where $g$ and $X$ are as above, and $\lambda\in\u^*(n,\K)$. The left 
and right actions commute, and together they give a double action of $U(n,\K)$ on
$\u^*(n,\K)$. Moreover
\[
g*\lambda *g^{-1}=\lambda^g.
\]
The {\it coadjoint orbit} $\Omega(\lambda)$ and the {\it coadjoint cluster} $\Psi(\lambda)$ of $\lambda$ are
\begin{eqnarray*}
\Omega(\lambda) & = & \left\{\, g* \lambda *g^{-1}:\quad g\in U(n,\K)\,\right\}, \\
\Psi(\lambda) & = & \left\{\, g* \lambda * g': \quad g,\,g'\in U(n,\K)\, \right\},
\end{eqnarray*}
with $\Omega(\lambda)\subset\Psi(\lambda)$.
Each coadjoint cluster is a union of coadjoint orbits.

There is a natural correspondence between adjoint orbits in $\u(n,\K)$ and conjugacy classes
in $U(n,\K)$ via the map $X \mapsto I+X$, by virtue of the identity 
$I+g\cdot X\cdot g^{-1}=g\cdot (I+X)\cdot g^{-1}$. Thus each adjoint cluster naturally corresponds
to a union of conjugacy classes, which we refer to as a {\it conjugacy cluster}.

We use $\C[U(n,\K)]$, $\C[\u(n,\K)]$ and
$\C[\u^*(n,\K)]$ to denote the respective spaces of
complex-valued functions on $U(n,\K)$, $\u(n,\K)$ and $\u^*(n,\K)$.

\section{Cluster modules and their characters}\label{module}

Let $\theta$ be a fixed non-trivial
homomorphism from the additive group of $\F_q$ to the multiplicative group of $\C$. The {\it Fourier basis} for $\C[U(n,\F_q)]$ consists of
the elements
\[
\left\{\, v(\lambda):\ \lambda\in\u^*(n,\F_q)\,\right\}
\]
defined by
\begin{eqnarray*}
v(\lambda):\quad g & \mapsto & \theta[\lambda(g-I)].
\end{eqnarray*}
This is an orthonormal basis for ${\mathbb C}[U(n,\F_q)]$ as a ${\mathbb C}$-linear space. Note that
\begin{equation}\label{3.0.1}
v(\lambda_1 + \lambda_2)\,=\, v(\lambda_1)\cdot v(\lambda_2)\,. 
\end{equation}

Consider the regular representation of $U(n,\F_q)$ on
${\mathbb C}[U(n,\F_q)]$ given by the left action. For 
$v\in{\mathbb C}[U(n,\F_q)]$ the left action of $g\in U(n,\F_q)$
sends $v$ to $g* v$ defined by
\begin{eqnarray*}
g* v:\quad g' & \mapsto & v(g'g).
\end{eqnarray*}
In a sense, one can say that the representation theory of a finite group is all about 
decomposing the regular representation. In the
case of $U(n,\F_q)$, it is particularly fruitful to consider the regular representation
in terms of the Fourier basis $\{\,v(\lambda)\,\}$.

\begin{proposition}\label{3.1}
\[
g* [v(\lambda)]=[v(\lambda)(g)]\cdot v(g*\lambda).
\]
\end{proposition}

\begin{proof}
The function $v(\lambda)$ sends $g'$ to $\theta[\lambda(g'-I)]$, so
$g* [v(\lambda)]$ sends $g'$ to $\theta[\lambda(g'g-I)]$. Now
\begin{eqnarray*}
{\theta[\lambda(g'g-I)]} & = & {\theta[\lambda(g'g-g+g-I)]} \\
{} & = & {\theta[\lambda(g-I)]}\cdot{\theta[\lambda((g'-I)\cdot g)]} \\
{} & = & {\theta[\lambda(g-I)]}\cdot{\theta[(g*\lambda)(g'-I)]} \\
{} & = & v(\lambda)(g)\cdot v(g*\lambda)(g').
\end{eqnarray*}
This gives the proposition. \end{proof}

The first of many consequences of this proposition is that for
any orbit $L$ in $\u^*(n,\F_q)$ under the left action of $U(n,\F_q)$, the $\C$-linear subspace in
${\mathbb C}[U(n,\F_q)]$ spanned by the set
\[
\left\{\, v(\lambda): \quad \lambda\in L \, \right\}
\]
is a submodule of the regular representation of $U(n,\F_q)$. We denote this submodule
by  $V(L)$, and let $\chi(L)$ be the trace of $U(n,\F_q)$ on $V(L)$.
The dimension of $\chi(L)$ is equal to the size of $L$, and
\begin{eqnarray*}
{\mathbb C}[U(n,\F_q)] & = & \bigoplus_{L} V(L)\,,
\end{eqnarray*}
where the sum is over all orbits in $\u^*(n,\F_q)$ under the left action of $U(n,\F_q)$.
Clearly $V(L)$ is generated by $v(\lambda)$ for any $\lambda\in L$, so we can write
$V(\lambda)=V(L)$ and $\chi(\lambda)=\chi(L)$. In this notation we have $V(\lambda)=V(g*\lambda)$.

\begin{definition}
We will refer to the modules $V(\lambda)$ as {\it cluster modules} and the
characters $\chi(\lambda)$ as {\it cluster characters}.
\end{definition}

\begin{remark}
The notation $\chi(0)$ naturally stands for the trivial character, which of course is the cluster character
associated to $\lambda=0$.
\end{remark}

\begin{proposition}\label{3.2}
The cluster modules $V(\lambda)$ and $V(\lambda *g)$ are isomorphic.
\end{proposition}

\begin{proof}
The isomorphism is given by the right action of $U(n,\F_q)$ on 
$\C[U(n,\F_q)]$. For
$v\in\C[U(n,\F_q)]$ the right action by $g$
sends $v$ to $v*g$ defined by
\begin{eqnarray*}
v*g:\quad g' & \mapsto & v(gg').
\end{eqnarray*}
Similar to the case of the left action in Proposition \ref{3.1}, we have
\[
[v(\lambda)] * g=[v(\lambda)(g)]\cdot v(\lambda *g)\,, 
\]  
so the right action by $g$ sends $V(\lambda)$ to $V(\lambda *g)$. \end{proof}

\begin{corollary}\label{3.3}
The cluster character $\chi(\lambda)$ only depends on the coadjoint cluster 
$\Psi$ to which $\lambda$ belongs. We can therefore write $\chi(\Psi)=\chi(\lambda)$.
\end{corollary}

\begin{theorem}\label{3.4}
The regular character of $U(n,\F_q)$ is equal to the sum
\[
\sum_{\Psi}\frac{|\Psi|}{\deg\chi(\Psi)}\cdot\chi(\Psi)
\]
over the collection of coadjoint clusters in $\u^*(n,\F_q)$.
\end{theorem}

\begin{proof}
For each cluster $\Psi$ let $V(\Psi)$ be the submodule in ${\mathbb C}[U(n,\F_q)]$
spanned by the set
\[
\left\{\, v(\lambda): \quad \lambda\in\Psi \, \right\}
\]
The dimension of $V(\Psi)$ is $|\Psi|$, and
\begin{eqnarray*}
{\mathbb C}[U(n,\F_q)] & = & \bigoplus_{\Psi} V(\Psi).
\end{eqnarray*}
By Proposition \ref{3.2}, $V(\Psi)$ is a direct sum of isomorphic copies of cluster modules, 
each copy corresponds to an orbit in $\Psi$ under the left action of $U(n,\F_q)$.
The trace of  $U(n,\F_q)$ on each copy is equal to $\chi(\Psi)$, and the 
number of such copies must be ${|\Psi|}/{\deg\chi(\Psi)}$.
We therefore have the decomposition in the theorem. \end{proof}

Let $L$ be an orbit in $\Psi$ under the left action of $U(n,\F_q)$. By Proposition \ref{3.1}
\begin{equation}\label{3.0.3}
\chi(\Psi)(g)=\sum_{g*\lambda=\lambda}v(\lambda)(g)\,,
\end{equation}
where the sum is over all elements of $L$ that are fixed under the left action
by $g$. 
The following theorem gives the character formula one would expect from the philosophy
of orbit method (see Kirillov \cite[Conjecture 2.2.1]{Kirillov-Variations}):

\begin{theorem}\label{3.5}
\[
\chi(\Psi)(g)=
\frac{\deg\chi(\Psi)}{|\Psi|}\cdot
\sum_{\lambda\in\Psi}\theta[\lambda(g-I)].
\]  
\end{theorem}

\begin{proof}
We can write equation \eqref{3.0.3} as
\[
\chi(\Psi)(g)=\sum_{\lambda\in L}\langle\, v(g*\lambda),\ v(\lambda)\, \rangle\cdot v(\lambda)(g)\,,
\]
or in other words
\[
\chi(\Psi)(g)=\frac{1}{|U(n, {\mathbb F_q})|}
\sum_{\lambda\in L}\sum_{g'}\left[v(g*\lambda-\lambda)(g')\cdot v(\lambda)(g)\right]\,,
\]
where the second sum is over all elements of $U(n, {\mathbb F_q})$.
The expression in the brackets can be simplified:
\begin{eqnarray*}
{} & {} & v(g*\lambda-\lambda)(g')\cdot v(\lambda)(g) \\
{} & = & \theta\left[(g*\lambda-\lambda)(g'-I)+\lambda(g-I)\right] \\
{} & = & \theta\left[\lambda(g'g-g)-\lambda(g'-I)+\lambda(g-I)\right] \\
{} & = & \theta\left[\lambda(g'g-g')\right]\\
{} & = & \theta\left[(\lambda *g')(g-I)\right]\,,
\end{eqnarray*}
so
\[
\chi(\Psi)(g)=\frac{1}{|U(n, {\mathbb F_q})|}
\sum_{\lambda\in L}\sum_{g'}\theta\left[(\lambda *g')(g-I)\right].
\]
If $\lambda$ runs through all elements of $L$ and $g'$ runs through all elements of
$U(n, {\mathbb F_q})$, then $\lambda *g'$ will run through the elements
of $\Psi$, each with multiplicity
\[
\frac{|U(n, {\mathbb F_q})|\cdot |L|}{|\Psi|}\,,
\]
so we have
\[
\chi(\Psi)(g)=
\frac{|L|}{|\Psi|}\cdot
\sum_{\lambda\in\Psi}\theta[\lambda(g-I)].
\]
Finally, since $|L|=\deg\chi(\Psi)$, we obtain
\[
\chi(\Psi)(g)=
\frac{\deg\chi(\Psi)}{|\Psi|}\cdot
\sum_{\lambda\in\Psi}\theta[\lambda(g-I)].
\]
\end{proof}

The character formula looks particularly elegant when we express it in
terms of the Fourier basis:
\begin{equation}\label{3.0.2}
\chi(\Psi)=
\frac{\deg\chi(\Psi)}{|\Psi|}\cdot
\sum_{\lambda\in\Psi}v(\lambda).
\end{equation}
Essentially, $\chi(\Psi)$ is obtained by averaging over the subset
$\{\,v(\lambda):\ \lambda\in\Psi\,\}$ of the basis. 
Note that the regular character of $U(n,\F_q)$ is equal to the sum
\[
\sum_{\lambda}v(\lambda)
\]
taken over the entire basis. From this point of view, the decomposition in Theorem
\ref{3.4} is completely transparent. 

\begin{corollary}\label{3.6}
$\chi(\Psi)(g)=\chi(\Psi)(g')$ whenever $g$ and $g'$ belong to the same conjugacy cluster.
\end{corollary}

\begin{corollary}\label{3.7}
The interwining number between two cluster characters $\chi(\Psi)$ and $\chi(\Psi')$ is zero
unless $\Psi' =\Psi$, in which case
\[
\langle\, \chi(\Psi),\ \chi(\Psi)\, \rangle =
\frac{[\deg\chi(\Psi)]^2}{|\Psi|}.
\]
\end{corollary}

\begin{definition}
We say a complex-valued function on $U(n,\F_q)$ is a {\it cluster function}
if it is constant over each conjugacy cluster.
\end{definition}

In Section \ref{cluster} we will classify the adjoint clusters
in $\u(n,\K)$ and the coadjoint clusters in $\u^*(n,\K)$,
and see that the two kinds of clusters are equal in number when $\K=\F_q$. As a consequence, we have
the following intrinsic characterization of the cluster characters:

\begin{corollary}\label{3.8}
The collection of all cluster characters form an orthogonal basis for the 
${\mathbb C}$-linear space of cluster functions on $U(n,\F_q)$. 
\end{corollary}

\section{Adjoint clusters and coadjoint clusters}\label{cluster}

In the first two parts of this section we classify the adjoint clusters
in $\u(n,\K)$ and the coadjoint clusters in $\u^*(n,\K)$,
where $\K$ is an arbitrary field. We then study various important aspects of the coadjoint clusters:
Section \ref{sec4.3} defines two numerical
invariants of coadjoint clusters, which correspond (when $\K=\F_q$) to the 
degree and the self-interwining number of cluster characters (see Section \ref{degree}).
Section \ref{sec4.4} introduces a {\it primary decomposition} for coadjoint clusters, 
which will be used (when $\K=\F_q$)
to derive a corresponding primary decomposition for cluster characters 
(see Section \ref{primary}).

\subsection{Classification of adjoint clusters}\label{sec4.1}

\begin{notation}
For $1\le i<j\le n$ let $e_{ij}$ be the matrix whose only non-zero entry is the $(i,j)$-entry, 
with value $1$. The collection
\[
\left\{\,e_{ij}: \quad 1\le i<j\le n\,\right\}
\]
form a $\K$-linear basis for $\u(n,\K)$. 
We denote the dual basis in $\u^*(n,\K)$ by
\[
\left\{\,\e_{ij}: \quad 1\le i<j\le n\,\right\}
\]
so that
\[
\e_{ij}(X)=x_{ij}
\]
where $X\in\u(n,\K)$ and $x_{ij}$ is the $(i,j)$-entry of $X$. 
\end{notation}

The left and right actions of $U(n,\K)$ on $\u(n,\K)$ are special
kinds of row and column operations, respectively. Each element of $U(n,\K)$ can
be written as a product
\[
\prod_{1\le i<j\le n}(I+a_{ij}\cdot e_{ij})
\]
where $a_{ij}\in\K$ and the factors are arranged in a suitable order. 
Left action by $I+a_{ij}\cdot e_{ij}$ corresponds to the row
operation which adds $a_{ij}$ times the $j$-th row to the $i$-th row.
A general left action is a sequence of such row operations.
Similarly, right action by $I+a_{ij}\cdot e_{ij}$ corresponds to the column
operation which adds $a_{ij}$ times the $i$-th column to the $j$-th column, and
a general right action is a sequence of such column operations.

\begin{definition}
We say $X\in\u(n,\K)$ is an {\it adjoint template} if it has
at most one non-zero entry in each row and each column.
Let $\T(n,\K)$ denote the collection of adjoint templates in $\u(n,\K)$.  
\end{definition}

\begin{theorem}\label{4.1}
Each adjoint cluster in $\u(n,\K)$ contains one and only one element
in $\T(n,\K)$. The adjoint clusters are therefore indexed by the adjoint templates.
\end{theorem}

\begin{proof}
To reduce an arbitrary element in $\u(n,\K)$ to some element in
$\T(n,\K)$, we proceed column by column, from left to right. Suppose we have
successfully reduced the first $k$ columns, so that no two non-zero entries in 
the first $k$ columns belong to the same row nor the same column. We now look at
the $(k+1)$-th column. By a suitable column operation we can eliminate
from this column each non-zero entry which shares the same row as another
non-zero entry to the left. These column operations are guaranteed not to interfere
with each other since none of the first $k$ columns contains more than one non-zero
entry. From the remaining non-zero entries in the $(k+1)$-th column we can
eliminate all but the bottom one (if there is one at all) by a sequence of row
operations. We are thus done with this column. The process may continue for the
remaining columns, until we end up with an element in $\T(n,\K)$.

We next show that no two adjoint templates belong to the same adjoint cluster. 
For an element $X$ in $\u(n,\K)$, let $\rank(i,j,X)$ be the rank of
the matrix
\[
\sum_{i\le k<l\le j}x_{kl}\cdot e_{kl}\,.
\]
If $X$ and $X'$ belong to the same adjoint cluster, then $
\rank(i,j,X)=\rank(i,j,X')$ for
any $i$ and $j$. This implies that if moreover $X$ and $X'$ are both in $\T(n,\K)$,
then the non-zero entries of $X$ and $X'$ must have exactly the same positions. To
show that $X=X'$, take $g$ and $g'$ in $U(n,\K)$
such that $g\cdot X=X'\cdot g'$. The non-zero entries in $X$ and $X'$ have the same
values as the entries at corresponding positions in $g\cdot X$ and $X'\cdot g'$,
respectively, and this means $X=X'$. \end{proof}

\subsection{Classification of coadjoint clusters}\label{sec4.2}

\begin{notation}
It is convenient to introduce a
bijective $\K$-linear map from $\u^*(n,\K)$ to $\u(n,\K)$, given by
\begin{eqnarray*}
\lambda & \mapsto & e(\lambda)=\sum_{1\le i<j\le n}\lambda(e_{ij})\cdot e_{ij}\,.
\end{eqnarray*}
In particular, $e(\e_{ij})=e_{ij}$.
The inverse map is 
\begin{eqnarray*}
X & \mapsto & \e(X)=\sum_{1\le i<j\le n}x_{ij}\cdot \e_{ij}\,.
\end{eqnarray*}
In particular, $\e(e_{ij})=\e_{ij}$.
\end{notation}

The left and right actions 
of $U(n,\K)$ on $\u^*(n,\K)$ can be described in terms
of $\{\, e_{ij}\,\}$ and $\{\, \e_{ij}\,\}$:
\[
(I+a_{ij}\cdot e_{ij})*\e_{kl}=
\begin{cases}
\e_{kl}+a_{ij}\cdot \e_{ki}, &\text{if $l=j$ and $k<i$;} \\
\e_{kl}, &\text{otherwise.}
\end{cases}
\]
\[
\e_{kl}*(I+a_{ij}\cdot e_{ij})=
\begin{cases}
\e_{kl}+a_{ij}\cdot \e_{jl}, &\text{if $k=i$ and $l>j$;} \\
\e_{kl}, & \text{otherwise.}
\end{cases}
\]
Left action by $I+a_{ij}\cdot e_{ij}$ on
$\lambda\in\u^*(n,\K)$ corresponds to a {\it restricted column operation}
on $e(\lambda)$, which adds
to each entry to the right of the diagonal in the $i$-th column $a_{ij}$ times the entry
to its right in the $j$-th column.
Similarly, right action by $I+a_{ij}\cdot e_{ij}$ on
$\lambda$ corresponds to a {\it restricted row operation}
on $e(\lambda)$, which adds  
to each entry above the diagonal in the $j$-th row $a_{ij}$ times the entry 
above it in the $i$-th row.

\begin{definition}
For $\lambda\in\u^*(n,\K)$ the {\it support} of $\lambda$, denoted by $\supp(\lambda)$, is the
set of positions $(i,j)$ such that $\lambda(e_{ij})\not=0$.
In other words, $\lambda$ is supported at the positions of the non-zero entries of $e(\lambda)$.
\end{definition}

\begin{definition}
We say $\lambda\in\u^*(n,\K)$ is a {\it coadjoint template} if $\supp(\lambda)$
contains at most one position in each row and each column. 
Let $\T^*(n,\K)$ denote the collection of coadjoint templates in $\u^*(n,\K)$.
\end{definition}

\begin{remark}
By definition, $\lambda$ is in $\T^*(n,\K)$ if and only if $e(\lambda)$ is in $\T(n,\K)$.
\end{remark}

\begin{theorem}\label{4.2}
Each coadjoint cluster in $\u^*(n,\K)$ contains one and only one element
in $\T^*(n,\K)$. The coadjoint clusters are therefore indexed by the coadjoint templates. 
\end{theorem}

\begin{proof}
For $\lambda\in\u^*(n,\K)$ we apply restricted row and column
operations on $e(\lambda)$ to reduce it to some element in $\T(n,\K)$. This can be done if we proceed
column by column, from right to left, in the manner described in the proof of Theorem \ref{4.1},
but replace ``left'' by ``right'' and ``bottom'' by ``top''
in the recipe. 

We now show that no two coadjoint templates belong to the same coadjoint cluster.
For an element $\lambda$ in $\u^*(n,\K)$, let $\rank(i,j,\lambda)$ be the rank of
the matrix
\[
\sum_{k\le i<j\le l}\lambda(e_{kl})\cdot e_{kl}
\]
If $\lambda$ and $\lambda'$ belong to the same coadjoint cluster, 
then $\rank(i,j,\lambda)=\rank(i,j,\lambda')$ for
any $i$ and $j$. This implies that if moreover $\lambda$ and $\lambda'$ are both in $\T^*(n,\K)$,
then they must have exactly the same support. 
To show that $\lambda=\lambda'$, take $g$ and $g'$ in $U(n,\K)$
such that $g*\lambda =\lambda'* g'$. The non-zero entries in $e(\lambda)$ and $e(\lambda')$ 
have the same values as the entries at corresponding positions in $e(g*\lambda)$ and 
$e(\lambda'*g')$, respectively, which means $e(\lambda)=e(\lambda')$, and $\lambda=\lambda'$. \end{proof}

\subsection{Two numerical invariants of coadjoint clusters}\label{sec4.3}
 
To each element of $\u^*(n,\K)$
we will now associate two numerical values that remain invariant under the double action of $U(n,\K)$. 

\begin{notation}
For $\lambda\in\u^*(n,\K)$, Let $L(\lambda)$ and $R(\lambda)$ be the respective
orbits of $\lambda$ under the left and right actions of $U(n,\K)$. Define
\begin{eqnarray*}
\L(\lambda) & = &  L(\lambda) - \lambda \,, \\
\R(\lambda) & = &  R(\lambda) - \lambda \,.
\end{eqnarray*}
It is easy to see that $\L(\lambda)$ and $\R(\lambda)$ are $\K$-linear subspaces of 
$\u^*(n,\K)$. By definition, $\lambda'$ is an element of $\L(\lambda)$ if we
can find $Y\in\u(n,\K)$ such that $\lambda'(X)=\lambda(XY)$ for all
$X\in\u(n,\K)$. Similarly, $\lambda'$ is an element of $\R(\lambda)$ if we
can find $Y\in\u(n,\K)$ such that $\lambda'(X)=\lambda(YX)$ for all
$X\in\u(n,\K)$.
\end{notation}

\begin{proposition}\label{4.7}
Suppose $\tau\in\T^*(n,\K)$. Then $\lambda\in\L(\tau)$ if and only if every position in $\supp(\lambda)$ lies to the
left of some position in $\supp(\tau)$. Similarly, 
$\lambda\in\R(\tau)$ if and only if every position in $\supp(\lambda)$ lies below some position in $\supp(\tau)$.
Consequently, $\dim\L(\tau)=\dim\R(\tau)$.
\end{proposition}

\begin{proposition}\label{4.6}
For $\lambda\in\u^*(n,\K)$, the dimensions of $\L(\lambda)$, $\R(\lambda)$ and $\L(\lambda)\cap\R(\lambda)$  
only depend
on the coadjoint cluster $\Psi$ to which $\lambda$ belongs. Moreover $\dim\L(\lambda)=\dim\R(\lambda)$.
\end{proposition}

\begin{proof}
It is easy to verify the following relations:
\[
\L(\lambda*g)=\L(\lambda)*g\,, \qquad \L(g*\lambda)=g*\L(\lambda)=\L(\lambda)\,, 
\]
\[
\R(g*\lambda)=g*\R(\lambda)\,, \qquad \R(\lambda*g)=\R(\lambda)*g=\R(\lambda)\,.
\]  
As a consequence, 
the dimensions of $\L(\lambda)$, $\R(\lambda)$ and $\L(\lambda)\cap\R(\lambda)$
are invariant under the double action of $U(n,\K)$, so
they are characteristics of the coadjoint cluster $\Psi$ to which $\lambda$ belongs. Let $\tau$ be the 
coadjoint template in $\Psi$. By Proposition \ref{4.7},  $\dim\L(\lambda)=\dim\L(\tau)=\dim\R(\tau)=\dim\R(\lambda)$.
\end{proof}
 
\begin{definition} 
We define the {\it dimension index} $d(\lambda)$ to be the dimension of $\L(\lambda)$ and $\R(\lambda)$, and define
the {\it interwining index} $i(\lambda)$ to be the dimension of the intersection $\L(\lambda)\cap\R(\lambda)$.
By Proposition \ref{4.6} we can write $d(\Psi)=d(\lambda)$ 
and $i(\Psi)=i(\lambda)$.
\end{definition}

It is rather easy to compute $d(\Psi)$ and $i(\Psi)$ from the coadjoint template in $\Psi$: 
Proposition \ref{4.7} says that if $\tau\in\T^*(n,\K)$ then
$d(\tau)$ and $i(\tau)$ can be determined from $\supp(\tau)$ by simply counting various kinds of positions above the diagonal.
To compute $d(\tau)$ we just count the number of positions $(i,j)$ such that $(i',j)\in\supp(\tau)$ for some $i'<i$,
or equivalently, we can count the number of positions $(i,j)$ such that $(i,j')\in\supp(\tau)$ for some $j'>j$. 
In other words,
$d(\tau)$ is equal to the total distance from its support
to the second diagonal (the diagonal right above the main diagonal). In
particular, $d(\tau)$ is zero if $\supp(\tau)$ is contained in the second diagonal. 
To compute $i(\tau)$ we count the number of positions $(i,j)$ such that there is simultaneously some $i'<i$
with $(i',j)\in\supp(\tau)$ and some $j'>j$ with $(i,j')\in\supp(\tau)$.
A more visual way to obtain $i(\tau)$
is to consider the three positions $(i',j)$, $(i,j)$ and $(i,j')$ as forming a L-shaped hook,
as in the following diagram:
\[
\begin{array}{ccc}
(i',j) & {} & {} \\
\uparrow & {} & {} \\
(i,j) & \rightarrow & (i,j')
\end{array}
\]
Here the corner position $(i,j)$ is above the diagonal, while the two arrows each points to a position in
$\supp(\tau)$. Since no two such hooks can share the same corner position, 
to obtain $i(\tau)$ we only need to count the number of such corners.

\subsection{Primary decomposition of coadjoint clusters}\label{sec4.4}

For any $\lambda_1$ and $\lambda_2$ in $\u^*(n,\K)$ we clearly have
\[ 
\Psi(\lambda_1+\lambda_2)\,\subset\,\Psi(\lambda_1)+\Psi(\lambda_2)\,=\,
\{\, \lambda_1'+\lambda_2':\quad (\lambda_1', \lambda_2')\in \Psi(\lambda_1)\times\Psi(\lambda_2) \,\}.
\]
We will see that if two coadjoint templates $\tau_1$ and $\tau_2$ in $\T^*(n,\K)$ are supported in different 
rows and different columns then $\Psi(\tau_1+\tau_2)=\Psi(\tau_1)+\Psi(\tau_2)$. 

\begin{notation}
For $1\le i<j\le n$ we let $\E_{ij}$ denote the coadjoint cluster in $\u^*(n,\K)$ that contains $\e_{ij}$.
\end{notation}

\begin{definition}
We say a coadjoint cluster in $\u^*(n,\K)$ is {\it primary} if it is the 
multiple of some $\E_{ij}$ by a non-zero element in $\K$.
\end{definition}

Theorem \ref{4.4} will show that any coadjoint cluster is equal to a sum 
of primary coadjoint clusters.

\begin{lemma}\label{4.3}
For $\lambda\in\u^*(n,\K)$, if $\supp(\lambda)$ does not contain
any position in the $i$-th row nor the $j$-th column, then the coadjoint 
cluster $\Psi(\lambda+a\cdot \e_{ij})$ is equal to $\Psi(\lambda)+a\cdot\E_{ij}$. 
\end{lemma}

\begin{proof}
The statement is trivial when $a=0$, so assume $a\not=0$. 
The elements of $a\cdot\E_{ij}$ are of the form
\[
a\cdot \e_{ij}+\sum_{i<k<j}a_k\cdot \e_{kj}+\sum_{i<l<j}a'_l\cdot \e_{il}+
\sum_{i<k<l<j}a^{-1}\cdot a_k\cdot a'_l\cdot \e_{kl}
\]
where $a_k$ and $a'_l$ are elements in $\K\,$.
The condition on $\supp(\lambda)$
makes it possible to bring $\lambda+a\cdot \e_{ij}$ to
\[
\lambda+a\cdot \e_{ij}+\sum_{i<k<j}a_k\cdot \e_{kj}+\sum_{i<l<j}a'_l\cdot \e_{il}+
\sum_{i<k<l<j}a^{-1}\cdot a_k\cdot a'_l\cdot \e_{kl}
\]
through a double action of $U(n,\K)$, so we see that $\lambda+a\cdot\E_{ij}$ is contained
in the cluster $\Psi(\lambda+a\cdot \e_{ij})$. For any $\lambda'\in\Psi(\lambda)$ we
can bring $\lambda+a\cdot\E_{ij}$ to $\lambda'+a\cdot\E_{ij}$ through a double action, so 
$\lambda'+a\cdot\E_{ij}$ is also in $\Psi(\lambda+a\cdot \e_{ij})$. The entire sum
$\Psi(\lambda)+a\cdot\E_{ij}$
is therefore contained in $\Psi(\lambda+a\cdot \e_{ij})$. Finally, the double action is linear, so the
sum is closed under the double action, and must be equal to $\Psi(\lambda+a\cdot \e_{ij})$.
\end{proof}

\begin{theorem}\label{4.4}
For any $\tau\in\T^*(n,\K)$,
\[
\Psi(\tau)\,=\sum_{1\le i<j\le n}\tau(e_{ij})\cdot \E_{ij}.
\]
\end{theorem}

\begin{proof}
We proceed by induction on the size of $\supp(\tau)$.
Suppose $\tau(e_{kl})$ is non-zero for some $k$ and $l$. Put 
$\tau'=\tau-\tau(e_{kl})\cdot \e_{kl}$,
then $\tau'$ is again an element of $\T^*(n,\K)$, 
with a smaller-sized support than $\tau$, so we can assume that
\[
\Psi(\tau')\,=\sum_{1\le i<j\le n}\tau'(e_{ij})\cdot \E_{ij}.
\]
Now $\supp(\tau')$ does not contain any position in
the $k$-th row nor the $l$-th column, so by Lemma \ref{4.3}
\[
\Psi(\tau)\,=\,\Psi(\tau')+\tau(e_{kl})\cdot\E_{kl}\,=\sum_{1\le i<j\le n}\tau(e_{ij})\cdot \E_{ij}.
\]
\end{proof}

\begin{corollary}
If $\tau_1$ and $\tau_2$ in $\T^*(n,\K)$ are supported in different 
rows and different columns then $\Psi(\tau_1+\tau_2)=\Psi(\tau_1)+\Psi(\tau_2)$. 
\end{corollary}

\section{The number of cluster characters}\label{number}

Cluster characters naturally correspond to coadjoint clusters, while conjugacy clusters naturally
correspond to adjoint clusters. From Theorems \ref{4.1} and \ref{4.2} we see that the 
same set of combinatorial data naturally index both the cluster characters and the conjugacy clusters.
This is analogous to the fact that the partitions of $n$ index both the irreducible characters and the conjugacy classes of the
symmetric group $S_n$.

The number $B(n,q)$ of cluster characters can be calculated by enumerating the 
elements in $\T(n,\F_q)$. For example
\begin{eqnarray*}  
B(1,q) & = & 1, \\
B(2,q) & = & q, \\
B(3,q) & = & 1+3(q-1)+(q-1)^2, \\
B(4,q) & = & 1+6(q-1)+7(q-1)^2+(q-1)^3.
\end{eqnarray*}
In general we have the following recurrence relation, which shows that $B(n,q)$ is a polynomial version of the Bell
numbers.

\begin{theorem}\label{5.1}
$B(n,q)$ is determined by the recurrence relation
\[
B(n+1,q)=\sum_{k=0}^n\binom{n}{k}(q-1)^{n-k}\cdot B(k,q)
\]
with $B(0,q)=1$.
\end{theorem}

\section{The degree and the self-interwining number}\label{degree}

In this section we express the degree and the self-interwining number of a cluster
character $\chi(\Psi)$ in terms of the dimension index $d(\Psi)$ and the interwining index $i(\Psi)$ 
introduced in Section \ref{sec4.3}.

We will need a well-known fact, stated in the following lemma:

\begin{lemma}\label{6.1}
Let $G$ and $H$ be finite groups. Let $S$ be a finite set on which $G$ has a left action and
$H$ has a right action. Then for any element $s$ of $S$,
\[
|GsH|=\frac{|Gs|\cdot |sH|}{|Gs\cap sH|}\,.
\]
\end{lemma}

\begin{theorem}\label{6.2}
For a coadjoint cluster $\Psi$ in $\u^*(n,\F_q)$, the cluster character $\chi(\Psi)$ has
degree $q^{d(\Psi)}$ and self-interwining number $q^{i(\Psi)}$. The size of
$\Psi$ is $q^{2d(\Psi)-i(\Psi)}$.
\end{theorem}

\begin{proof}
Let $\lambda\in\Psi$.
In the notations of Section \ref{4.4}, the lemma above says that
\[
|\Psi|=\frac{|L(\lambda)|\cdot |R(\lambda)|}{|L(\lambda)\cap R(\lambda)|}
=\frac{|\L(\lambda)|\cdot |\R(\lambda)|}{|\L(\lambda)\cap \R(\lambda)|}
=q^{2d(\Psi)-i(\Psi)}.
\]
The degree of $\chi(\Psi)$ is equal to the size of $L(\lambda)$, so
\[
\deg\chi(\Psi)=|L(\lambda)|=|\L(\lambda)|=q^{d(\Psi)}.
\]
By Corollary \ref{3.7} we have
\[
\langle\, \chi(\Psi),\ \chi(\Psi)\, \rangle =
\frac{[\deg\chi(\Psi)]^2}{|\Psi|}=q^{i(\Psi)}. 
\]
\end{proof}

\begin{corollary}\label{6.4}
$\chi(\Psi)$ is irreducible if and only if $i(\Psi)$ is zero.
\end{corollary}

The character formula in equation \eqref{3.0.2} now becomes
\begin{equation}\label{6.0.1}
\chi(\Psi)=q^{i(\Psi)-d(\Psi)}\cdot
\sum_{\lambda\in\Psi}v(\lambda)\,.
\end{equation} 
We can also restate Theorem \ref{3.4}:

\begin{corollary}\label{6.3}
The regular character of $U(n,\F_q)$ is equal to the sum
\[
\sum_{\Psi}q^{d(\Psi)-i(\Psi)}\cdot\chi(\Psi)
\]
over the collection of coadjoint clusters in $\u^*(n,\F_q)$.
\end{corollary}

\section{Primary decomposition of cluster characters}\label{primary}

In this section we use the primary decomposition for coadjoint clusters established in Section \ref{sec4.4} to
derive a corresponding decomposition for cluster characters.

\begin{definition}
We say a cluster character is {\it primary} if it is associated to a primary coadjoint cluster.
By definition, $\chi(a\cdot\e_{ij})$ is primary when $a$ is any non-zero element of $\F_q$.
By Corollary \ref{6.4} the primary cluster characters are irreducible. 
\end{definition}

We now show that each cluster character can be decomposed into a tensor product of primary cluster characters.

\begin{theorem}\label{7.1}
For $\tau\in\T^*(n,\F_q)$,
\[
\chi(\tau)=\prod_{1\le i<j\le n}\chi(\tau(e_{ij})\cdot\e_{ij}).
\]
\end{theorem}

\begin{proof}
Write $\Psi_{ij}(\tau)=\tau(e_{ij})\cdot\E_{ij}$ and $\chi_{ij}(\tau)=\chi(\tau(e_{ij})\cdot\e_{ij})$.
The character formula in equation \ref{3.0.2} says
\begin{eqnarray}
\chi(\tau) &=&
\frac{\deg\chi(\tau)}{|\Psi(\tau)|}\cdot
\sum_{\lambda\in\Psi(\tau)}v(\lambda)\,, \label{7.1.1}\\
\chi_{ij}(\tau) &=&
\frac{\deg\chi_{ij}(\tau)}{|\Psi_{ij}(\tau)|}\cdot
\sum_{\lambda\in\Psi_{ij}(\tau)}v(\lambda)\,.
\end{eqnarray}
It is easy to verify from Theorem \ref{6.2} that
\begin{eqnarray}
\deg\chi(\tau) &=&
\prod_{1\le i<j\le n}\deg\chi_{ij}(\tau) \,, \\
q^{i(\tau)}\cdot|\Psi(\tau)| &=&
\prod_{1\le i<j\le n}|\Psi_{ij}(\tau)| \,. \label{7.1.4} 
\end{eqnarray}
From Theorem \ref{4.4} we know that each element $\lambda$ of $\Psi(\tau)$ can be decomposed
into a sum
\[
\lambda=\sum_{1\le i<j\le n}\lambda_{ij}
\]
where $\lambda_{ij}$ belongs to $\Psi_{ij}(\tau)$. If $\lambda'=g*\lambda *g'$ then each
such decomposition of $\lambda$ gives a corresponding decomposition of $\lambda'$:
\[
\lambda'=\sum_{1\le i<j\le n}g*\lambda_{ij}*g'
\]
so the number of such decompositions is the same for any element of $\Psi(\tau)$. We see
from \eqref{7.1.4} that this number must be $q^{i(\tau)}$. By equation \eqref{3.0.1} each decomposition of 
$\lambda$ gives a corresponding decomposition of $v(\lambda)$:  
\[
v(\lambda)=\prod_{1\le i<j\le n}v(\lambda_{ij})
\]
so
\begin{equation}\label{7.1.5}
q^{i(\tau)}\cdot\sum_{\lambda\in\Psi(\tau)}v(\lambda)=
\prod_{1\le i<j\le n}\left[\sum_{\lambda\in\Psi_{ij}(\tau)}v(\lambda)\right] \,.
\end{equation}
The theorem now follows if we combine \eqref{7.1.1}--\eqref{7.1.5}. 
\end{proof}

\begin{corollary}\label{7.3}
If $\tau_1$ and $\tau_2$ in $\T^*(n,\F_q)$ are supported in different rows and different columns then 
$\chi(\tau_1)\otimes\chi(\tau_2)=\chi(\tau_1+\tau_2)$.
\end{corollary}

\section{Tensor products of cluster characters}\label{tensor}

We know from Section \ref{module} that cluster characters are cluster functions. 
The tensor product of two cluster characters is again a cluster function, and so 
by Corollary \ref{3.8} it can be decomposed into a ${\mathbb C}$-linear combination of cluster
characters. We will see in this section that the coefficients are always non-negative
integers, so that the product is actually equal to a sum of cluster characters.

For coadjoint clusters $\Psi_1$, $\Psi_2$ and $\Psi$ in $\u^*(n, \F_q)$, the character formula in 
equation \eqref{6.0.1} gives the following expression for the 
interwining number between $\chi(\Psi_1)\otimes\chi(\Psi_2)$ and $\chi(\Psi)$:

\begin{proposition}\label{8.0}
\[
\langle\, \chi(\Psi_1)\otimes\chi(\Psi_2),\ \chi(\Psi)\, \rangle
=\frac{q^{i(\Psi_1)+i(\Psi_2)+i(\Psi)}}{q^{d(\Psi_1)+d(\Psi_2)+d(\Psi)}}\cdot
C(\Psi_1,\Psi_2,\Psi)\,,
\]
where $C(\Psi_1,\Psi_2,\Psi)$ is the size of the set 
\[
\left\{\, (\lambda_1,\lambda_2)\in\Psi_1\times\Psi_2
\ \mid\ \lambda_1+\lambda_2\in\Psi \,\right\}\,.
\]  
\end{proposition}

\begin{corollary}\label{8.1}
For coadjoint clusters $\Psi_1$ and $\Psi_2$ in $\u^*(n,\F_q)$, 
\[
\chi(\Psi_1)\otimes\chi(\Psi_2)=
\frac{q^{i(\Psi_1)+i(\Psi_2)}}{q^{d(\Psi_1)+d(\Psi_2)}}\cdot
\sum_{\Psi}q^{-d(\Psi)}\cdot
C(\Psi_1,\Psi_2,\Psi)\cdot\chi(\Psi)
\]
where the sum is over all coadjoint clusters in $\u^*(n,\F_q)$. 
\end{corollary}

Corollary \ref{8.1} says that the tensor product $\chi(\Psi_1)\otimes\chi(\Psi_2)$ is a linear combination of cluster characters
with non-negative rational coefficients. We will see in Theorem \ref{8.6} that the coefficients are in fact
integers.
In general it is a non-trivial task to determine the coefficients, with the tricky part being the computation of
$C(\Psi_1,\Psi_2,\Psi)$. The following is a very special case:

\begin{corollary}\label{8.7}
The multiplicity of $\chi(0)$ in $\chi(\Psi)\otimes\chi(\Psi')$ is zero unless $\Psi'=-\Psi$, in which case the multiplicity
is $q^{i(\Psi)}$.
\end{corollary}

From Section \ref{primary} we know that each cluster character can be decomposed into a tensor product of
primary cluster characters. In order to understand tensor  
products of cluster characters in general, we shall first consider the case when both cluster characters involved
are primary.  

\begin{notation}
For $1\le i<j\le n$ and $a\in\F_q$ we use $\chi(i,j,a)$ as an alternative notation for $\chi(a\cdot\e_{ij})$.
Note that $\chi(i,j,0)$ is the trivial character for any $i$ and $j$.  
\end{notation}

The next four propositions give the decomposition of $\chi(i,j,a)\otimes\chi(i',j',b)$.
We assume $ab\not=0$. There are the following cases to consider:
\begin{itemize}
\item[(1)] $i\ne i'$, $j\ne j'$,
\item[(2)] $i\ne i'$, $j=j'$,
\item[(3)] $i=i'$, $j\ne j'$,
\item[(4)] $i=i'$, $j=j'$, $a+b\ne 0$,
\item[(5)] $i=i'$, $j=j'$, $a+b=0$.
\end{itemize}

Case (1) is directly covered by Corollary \ref{7.3}. The product is again a cluster character.

In case (2), the positions $(i,j)$ and $(i',j)$ are in the same column. It appears that 
the factor associated to the higher position dominates the product, the other factor dissolves
away. This case can be reduced to case (1):

\begin{proposition}\label{8.2}
If $ab\not=0$ and $i'>i$ then $\chi_(i,j,a)\otimes\chi(i',j,b)$ is equal to 
\[
\chi(i,j,a)\otimes
\left[\chi(0)+\sum_{i'<j'<j}\sum_{c\ne 0}\chi(i',j',c)\right]\,.
\]
\end{proposition}

Case (3) is similar to case (2). The two positions $(i,j)$ and $(i,j')$ are in the same row,
the factor associated to the position on the right dominates, the other factor dissolves away, and the
case can be reduced to case (1):

\begin{proposition}\label{8.3}
If $ab\ne 0$ and $j'<j$ then $\chi(i,j,a)\otimes\chi(i,j',b)$ is equal to
\[
\chi(i,j,a)\otimes
\left[\chi(0)+\sum_{i<i'<j'}\sum_{c\ne 0}\chi(i',j',c)\right]\,.
\]
\end{proposition}

In cases (4) and (5), the two positions coincide, and the factors play symmetric roles.
Case (4) can be reduced to case (2) or case (3), while case (5) can be directly reduced to case (1).

\begin{proposition}\label{8.4}
If $ab\ne 0$ and $a+b\ne 0$ then $\chi(i,j,a)\otimes\chi(i,j,b)$ is equal to
\[
\chi(i,j,a+b)\otimes
\left[\chi(0)+\sum_{i<i'<j}\,\sum_{c\ne 0}\chi(i',j,c)\right]\,,
\]
this is also equal to
\[
\chi(i,j,a+b)\otimes
\left[\chi(0)+\sum_{i<j'<j}\,\sum_{c\ne 0}\chi(i,j',c)\right]\,.
\]
\end{proposition}

\begin{proposition}\label{8.5}
If $a\ne 0$ then $\chi(i,j,a)\otimes\chi(i,j,-a)$ is equal to
\[
\left[\chi(0)+\sum_{i<i'<j}\sum_{b\ne 0}\chi(i',j,b)\right]\otimes
\left[\chi(0)+\sum_{i<j'<j}\sum_{c\ne 0}\chi(i,j',c)\right]\,.
\]
\end{proposition}

\begin{theorem}\label{8.6}
The tensor product of any finite number of cluster characters is equal to a linear combination
of cluster characters with non-negative integer multiplicities.
\end{theorem}

\begin{proof}
We use induction on the total degree of the tensor product. By Theorem 
\ref{7.1} we can write the product in the form $\chi(\tau_1)\otimes\chi(\tau_2)\otimes\cdots$,
where each factor is associated to a primary coadjoint template. 
if
all the templates involved are supported in different rows and different columns then
Corollary \ref{7.3} says the product is equal to the cluster character associated to $\tau_1+\tau_2+\cdots$.
If any two of the templates
are supported in the same row or column, with one (or both) supported
on the second diagonal, then the corresponding characters can be combined by 
applying the following special cases of Propositions \ref{8.2}-\ref{8.5}:
If $i<j$ then $\chi(i,i+1,a)\otimes\chi(i,j+1,b)=\chi(i,j+1, b)$ 
and $\chi(j,j+1,a)\otimes\chi(i,j+1,b)=\chi(i,j+1, b)$; moreover
$\chi(i,i+1,a)\otimes\chi(i,i+1,b)=\chi(i,i+1, a+b)$. After this, 
if we can still find two of the templates with supports in the same row or column, but both
above the second diagonal, then we can decompose the product of the corresponding characters
according to 
Propositions \ref{8.2}-\ref{8.5}. This will break the original product into pieces, each with a smaller
degree, and the induction works.
\end{proof}  

\begin{remark}
Theorem \ref{8.6} says that the collection of all $\Z$-linear combinations of   
cluster characters form a subring in the representation ring of $U(n,\F_q)$.
This subring has a rather pleasant structure: it is generated by the primary
cluster characters, with the relations given in Propositions \ref{8.2}--\ref{8.5}.
\end{remark}

\section{The discrete series character}\label{discrete}

Let $\Delta(n,\K)$ be the collection of elements $\lambda$ in $\u^*(n,\K)$
such that $\supp(\lambda)$ contains at least one position in each of the first
$n-1$ rows. It is easy to check that $\Delta(n,\K)$ is 
closed under the left action of $U(n,\K)$ on $\u^*(n,\K)$.
Thus by Proposition \ref{3.1} 
the subspace in ${\mathbb C}[U(n,\F_q)]$ spanned by the set
\[
\left\{\, v(\lambda):\quad \lambda\in\Delta(n,\F_q)\,\right\}
\]
is a submodule of the regular representation of $U(n,\F_q)$. We denote this submodule
by $D(n,\F_q)$, and let $\delta(n,\F_q)$ be the trace of $U(n,\F_q)$ on $D(n,\F_q)$.
The cluster module $V(\lambda)$ is contained in  $D(n,\F_q)$ for every 
$\lambda\in\Delta(n,\F_q)$, and $D(n,\F_q)$ is a direct sum of cluster modules. The character
$\delta(n,\F_q)$ is therefore a sum of cluster characters.

Theorem \ref{9.1} will show that $\delta(n,\F_q)$ is equal to the 
character of the
restriction to $U(n,\F_q)$ of any irreducible 
discrete series representation of $GL(n,\F_q)$ (see \cite[Section 3]{Lehrer74}). 
We therefore refer to $\delta(n,\F_q)$ as the {\it discrete series character} of $U(n,\F_q)$.

\begin{proposition}\label{9.4}
Suppose $\lambda\in\u^*(n,\K)$ and $X\in\T(n,\K)$, and write $g=I+X$. Then
$g*\lambda=\lambda$ if and only if 
$\supp(\lambda)$ contains no position above any non-zero entry of $X$. 
\end{proposition}

\begin{theorem}\label{9.1}
The discrete series character $\delta=\delta(n,\F_q)$ has value 
\[
\delta(g)=(-1)^{r(g)}\cdot (q-1)\cdot (q^2-1)\cdots (q^{n-1-r(g)}-1)
\]
at $g\in U(n,\F_q)$, where $r(g)$ is the rank of $g-I$.
\end{theorem}

\begin{proof}
Write $g=I+X$.
Since $\delta$ is a sum of cluster characters, it must be a cluster function, therefore
(since the right-hand side of the formula is a cluster function in the first place)
we can assume $X\in\T(n,\F_q)$. By Proposition \ref{3.1}
\[
\delta(g)=\sum_{g*\lambda=\lambda}v(\lambda)(g)\,,
\]
where the sum is over all $\lambda\in\Delta(n,\F_q)$ that are fixed under
the left action by $g$.

For $1\le k\le n-1$ let $\Lambda_k(X)$ be the collection of elements $\lambda$ in
$\u^*(n,\F_q)$ such that $\supp(\lambda)$ contains at least one position in the $k$-th row,
and no position in any other row or above any non-zero entry of $X$.
By Proposition \ref{9.4} 
an element $\lambda$ in $\Delta(n,\F_q)$ that is fixed under
the left action by $g$ can be written (uniquely)
as a sum
\[
\lambda=\sum_{k}\lambda_k
\]
where $\lambda_k\in\Lambda_k(X)$. We can therefore write 
\begin{equation}\label{9.1.1}
\delta(g)=\prod_{k}\left[\sum_{\lambda\in\Lambda_k(X)}v(\lambda)(g)\right]
=\prod_{k}\left[\sum_{\lambda\in\Lambda_k(X)}\theta[\lambda(X)]\right].
\end{equation}
Since $X\in\T(n,\F_q)$ each row of $X$ contains at most one
non-zero entry. By a simple computation we find
\begin{equation}\label{9.1.2}
\sum_{\lambda\in\Lambda_k(X)}\theta[\lambda(X)]=\begin{cases}
|\Lambda_k(X)|, &\text{if the $k$-th row of $X$ is zero;} \\
-1, &\text{if the $k$-th row of $X$ is non-zero.}
\end{cases}
\end{equation} 
From $X\in\T(n,\F_q)$ it is also easy to conclude that
\begin{equation}\label{9.1.3}
|\Lambda_k(X)|=q^{n-k-b_k(X)}-1\,, 
\end{equation}
where $b_k(X)$ is the number of non-zero entries of $X$ below the $k$-th row. 
Moreover, the rank of $X$ is equal to the number of its non-zero entries. The theorem now follows
if we combine \eqref{9.1.1}--\eqref{9.1.3}. 
\end{proof}

We next compute the multiplicity of each cluster character in the discrete series character.
We will need the following modified version of Lemma \ref{6.1}:

\begin{lemma}\label{9.2}
Let $G$ and $H$ be finite groups. Let $S$ be a finite set on which $G$ has a left action and
$H$ has a right action. Let $S'$ be a subset of $S$ which is closed under the action of $G$.
Then for any element $s$ of $S$,
\[
|GsH\cap S'|=\frac{|Gs|\cdot |sH\cap S'|}{|Gs\cap sH|}\,.
\]
\end{lemma}

For $\tau\in\T^*(n,\K)$ let $\R(\tau)$ be as in Section \ref{sec4.3}.
For $1\le k\le n-1$ let $\R_k(\tau)$ be the collection of elements $\lambda\in\R(\tau)$ such that
$\supp(\lambda)$ contains only positions (if any) in the $k$-th row. Clearly $\R_k(\tau)$ is a subspace
of $\R(\tau)$, and
\begin{eqnarray*}
\R(\tau) & = & \bigoplus_k \R_k(\tau) \,.
\end{eqnarray*}
Let $d(k,\tau)$ be the dimension of $\R_k(\tau)$. Then
\begin{eqnarray}\label{9.3.2}
d(\tau) & = & \sum_k d(k,\tau) \,.
\end{eqnarray}
By Proposition \ref{4.7}, $d(k,\tau)$ is equal to the number of positions $(i,j)\in\supp(\tau)$ such that $i<k<j$.

\begin{theorem}\label{9.3}
The discrete series character $\delta(n,\F_q)$ is equal to a sum of cluster characters. For
$\tau\in\T^*(n,\F_q)$ the multiplicity of $\chi(\tau)$ in $\delta(n,\F_q)$ is equal to
\[
q^{d(\tau)-i(\tau)}\cdot\prod_{k\in Z(\tau)}\left[ 1-\frac{1}{q^{d(k,\tau)}}\right]\,,
\]
where $Z(\tau)$ is the set of indices $1\le k\le n-1$ such that $\supp(\tau)$
contains no position in the $k$-th row.
\end{theorem}

\begin{proof}
The fact that $\delta(n,\F_q)$ is equal to a sum of cluster characters is explained after its definition
and already used in the proof of Theorem \ref{9.1}. Let $\Psi(\tau)$ be the coadjoint cluster of $\tau$,
and let $L(\tau)$ and $R(\tau)$ be the respective orbits of $\tau$ under the left and right actions of $U(n,\F_q)$. The 
multiplicity of $\chi(\tau)$ in $\delta(n,\F_q)$ is equal to the number of orbits in
$\Psi(\tau)\cap\Delta(n,\F_q)$ under the left action, and this number is equal to $|\Psi(\tau)\cap\Delta(n,\F_q)|/|L(\tau)|$.
By Lemma \ref{9.2},
\[
|\Psi(\tau)\cap\Delta(n,\F_q)|=\frac{|L(\tau)|\cdot |R(\tau)\cap\Delta(n,\F_q)|}{|L(\tau)\cap R(\tau)|}\,,
\]
so the multiplicity is equal to $|R(\tau)\cap\Delta(n,\F_q)|/|L(\tau)\cap R(\tau)|$. We know from Section \ref{sec4.3} that
$|L(\tau)\cap R(\tau)|=q^{i(\tau)}$, so the theorem is reduced to
\begin{equation}\label{9.3.1}
\qquad |R(\tau)\cap\Delta(n,\F_q)|=
q^{d(\tau)}\cdot\prod_{k\in Z(\tau)}\left[ 1-\frac{1}{q^{d(k,\tau)}}\right].
\end{equation}
We can write each element $\lambda\in R(\tau)$ as a sum
\[
\lambda=\tau+\sum_k\lambda_k 
\]
where $\lambda_k\in\R_k(\tau)$. Then $\lambda\in\Delta(n,\F_q)$ if and only if $\lambda_k\not=0$ for every
$k\in Z(\tau)$. Thus the number of elements in $R(\tau)\cap\Delta(n,\F_q)$ is equal to
\[
\left[\prod_{k\in Z(\tau)}(q^{d(k,\tau)}-1)\right]\cdot\left[\prod_{k\not\in Z(\tau)}q^{d(k,\tau)}\right]\,,
\]
and by \eqref{9.3.2} this is equal to
\[
q^{d(\tau)}\cdot\prod_{k\in Z(\tau)}\left[ 1-\frac{1}{q^{d(k,\tau)}}\right],
\]
so we obtain \eqref{9.3.1}. \end{proof}

\begin{remark}
It is interesting to compare Theorem \ref{9.3} with Corollary \ref{6.3}.
\end{remark}

We say $\tau\in\T^*(n,\K)$ is {\it degenerate} if $e(\tau)$ can be written in block diagonal form
\[
e(\tau)=\left(\begin{array}{cc}
e(\tau_1) & 0 \\
{} & e(\tau_2)
\end{array}\right)
\]
with $\tau_1\in\T^*(k,\K)$ and $\tau_2\in\T^*(n-k,\K)$ for some $1\le k\le n-1\,$.  
Equivalently, $\tau$ is degenerate if $d(k,\tau)=0$ for some $k\in Z(\tau)$.

\begin{corollary}
For $\tau\in\T^*(n,\F_q)$ the cluster character $\chi(\tau)$ is contained in 
the discrete series character if and only if $\tau$ is non-degenerate.
\end{corollary}

\section{Appendix}

\subsection{The character formula}

Arias-Castro, Diaconis and Stanley \cite[Theorem 2.2]{super-class-walk} 
and Andr\'e \cite[Theorem 3]{Carlos02}
gave a remarkable closed
form formula for the cluster characters. We can not resist the temptation to give this
formula a new formulation and a new proof in the framework of this paper.

For a coadjoint template $\tau\in\T^*(n,\K)$ and an adjoint template $X\in\T(n,\K)$,
define the index $i(X,\tau)$ to be the number of $\Gamma$-shaped hooks as in the following
diagram:
\[
\begin{array}{ccc}
(i,j) & \rightarrow & (i,j')\\
\downarrow & {} & {} \\
(i',j) & {} & {} 
\end{array}
\]
where the right arrow points to a position in $\supp(\tau)$ while the down arrow points to the
position of a non-zero entry in $X$.

\begin{theorem}
For a coadjoint template $\tau\in\T^*(n,\F_q)$ and an adjoint template $X\in\T(n,\F_q)$,
the cluster character $\chi(\tau)$ has the following value at $I+X$:
\[
\chi(\tau)(I+X)\,=\,\begin{cases}
q^{d(\tau)-i(X,\tau)}\cdot \theta[\tau(X)], 
& \text{if $X$ has no non-zero entry below or} \\
& \text{to the left of any position in $\supp(\tau)$;} \\
0, & \text{otherwise.}
\end{cases}
\]
\end{theorem}

\begin{proof}
Write $g=I+X$, and let $L(\tau)$ be the orbit of $\tau$ under the left action of $U(n,\F_q)$.
Equation \eqref{3.0.3} says
\[
\chi(\tau)(g)\,=\, \sum_{\lambda} v(\lambda)(g) \,=\,\sum_{\lambda} \theta[\lambda(X)],
\]
where the sum is over all elements $\lambda\in L(\tau)$ which are fixed under the left action of $g$.
There are three cases to consider:
\begin{itemize}
\item[(1)]
$X$ has a non-zero entry below some position in $\supp(\tau)$.
\item[(2)]
$X$ has a non-zero entry to the left of some position in $\supp(\tau)$.
\item[(3)]
$X$ has no non-zero entry below or to the left of any position in $\supp(\tau)$.
\end{itemize}
In case (1), Propositions \ref{4.7} and \ref{9.4} imply that no $\lambda\in L(\tau)$ can satisfy
$g*\lambda=\lambda$, thus $\chi(\tau)(g)=0$. In cases (2) and (3), by the same two propositions
we can write
\[
\chi(\tau)(g)\,=\, \theta[\tau(X)]\cdot\sum_{\lambda} \theta[\lambda(X)],
\]
where the sum is over those $\lambda\in\u^*(n, \F_q)$ such that $\supp(\lambda)$ contains only positions
(if any) to the left of $\supp(\tau)$ and no position above any non-zero entry of $X$. Note that the
number of $\lambda$'s with this property is precisely $q^{d(\tau)-i(X,\tau)}$. It is straightforward
to verify that the sum is zero in case (2) and $q^{d(\tau)-i(X,\tau)}$ in case (3). So we obtain the formula.
\end{proof}


\subsection{Supercharacter theory for finite groups}

A supercharacter theory for a finite group $G$ (as introduced by Diaconis and Isaacs in \cite{Supercharacters})
is defined by a collection $\mathcal{K}$ of non-empty subsets in $G$ and a collection $\mathcal{X}$ of 
complex-valued functions on $G$, with the following properties:
\begin{itemize}
\item[(1)]
The members of $\mathcal{K}$ (called {\it superclasses}) are mutually disjoint, each is invariant under conjugation by $G$, and their union
is equal to $G$.
\item[(2)]
The members of $\mathcal{X}$ (called {\it supercharacters}) are mutually orthogonal, each is the character of a complex linear representation of
$G$, and their sum is equal to the regular character of $G$.
\item[(3)]
Each supercharacter is constant on each superclass. 
\item[(4)]
$|\mathcal{X}|=|\mathcal{K}|$.
\end{itemize}  

The cluster method introduced in this paper leads to a specific supercharacter theory for
$U(n,\F_q)$, in which the superclasses are the conjugacy clusters and the supercharacters are (up to constant multiples) the
cluster characters. Clearly supercharacter theory has a much broader scope. It is an interesting project
to find other natural instances of supercharacter theory.

\section*{Acknowledgements}

This paper is based on my Ph.D. dissertation at the University of Pennsylvania. 
My deepest gratitude goes to Alexandre Kirillov, who kindly took me up as his student, generously
passed on to me his fascination with the triangular matrices, and continuously 
provided me with inspiration. I am also deeply grateful to Alexei Borodin, Persi Diaconis and Chunwei Song 
for their help, support and encouragement.

\nocite{*}
\bibliographystyle{abbrv}
\bibliography{cluster}
\pagebreak

\end{document}